\documentclass[a4paper,11pt,reqno]{amsart}

\usepackage{xcolor}
\usepackage{amsmath,amssymb}
\usepackage{enumerate}
\usepackage{ifthen}
\usepackage{calc}
\usepackage{hyperref}
\nonstopmode\numberwithin{equation}{section}
\newtheorem{definition}{Definition}[section]
\newtheorem{theorem}{Theorem}[section]

\newtheorem{lemma}{Lemma}[section]

\newtheorem{remark}{Remark}[section]

\allowdisplaybreaks

 \theoremstyle{plain}

\begin{document}
\title{ Fuzzy Generalizations Of Automorphism and Inner Automorphism of Groups}
\maketitle
 \begin{center}
\center{  $\text{Shiv Narain}^1, \text{Sunil Kumar}^2, \text{Sandeep Kumar}^2$ and ${\text{Gaurav Mittal}^{2, 3}}$ }
\medskip
{\footnotesize

 \center{$^1$Department of Mathematics, Arya P. G. College, Panipat, India\\$^{2}$DRDO, India \\$^{3}$Department of Mathematics, Indian Institute of Technology Roorkee, Roorkee, India\\ emails: drnarainshiv@gmail.com, sunilsangwan6174@gmail.com\\san.kuwar@gmail.com, gaurav.mittaltwins@gmail.com

 }
 }

\bigskip
\begin{abstract} The fuzzification of various concepts from classical set theory came into existence with the discovery of fuzzy set as a generalization of crisp set. This paper  coins the terms like fuzzy permutations on a set, fuzzy endomorphism, fuzzy automorphism and fuzzy inner automorphisms of groups. These new concept can be taken as fuzzy generalizations of usual groups theoretic permutations, endomorphism, automorphisms and inner automorphisms respectively. Notions of fuzzy permutations and fuzzy inner automorphism induced by a normal fuzzy subgroup are also introduced. Finally, we obtain the fuzzy analogues of well known theorems from classical group theory.
 
\end{abstract}
\end{center}
\hspace{-5mm}\subjclass {}{\textbf{AMS Subject Classifications}: 20N25}\\
\hspace{-5mm}
\section{Introduction}

The study of fuzzy set was initiated by Zadeh $[1]$, since then a large number of mathematical structures like algebras, topological spaces, differential equations etc. has been fuzzified by many mathematicians. Rosenfield $[2]$ introduced the concept of fuzzy subgroup with the assumption that  subsets of the group are fuzzy. With this evolution of fuzzy group theory, various fuzzy counterparts of group theoretic concepts from classical group~~ theory were introduced by many authors. To list a few in this context, Bhattacharya and Mukharjee $[3]$ introduced the concept of fuzzy cosets and fuzzy normal subgroups and proved fuzzy analogues of various group theoretic concepts. In $[4]$,  Mukharjee and Bhattacharya  introduced notion of order of a fuzzy subgroup in a finite group, fuzzy abelian group and fuzzy solvable group. Based on fuzzy binary operations, Yuan and Lee $[5]$, proposed a new kind of fuzzy group. Yao Bingxue $[6]$ introduced the concept of fuzzy homomorphism. Motivated by this, we introduced the concept of fuzzy automorphism, fuzzy inner automorphism in groups, and fuzzy analogue of some standard results from classical group theory.

\section{Preliminaries} 
\begin{definition}Let $X$ be a non empty set. Then $\mu:X \to [0, 1]$ is called a fuzzy subset of $X$. 

\end{definition}
We denote $\mathcal{F}\mathcal{P}(X)$, the set of all fuzzy subsets of $X$ also termed as fuzzy power set of $X$.
\begin{definition}Fuzzy map: 
Let $X$ and $Y$  be two non empty sets. Then a fuzzy subset of $X \times Y$ given by $f:X\times Y \to [0, 1]$ is called a fuzzy map if for each $x \in X$, there exists a unique $y_x \in Y$ such that $f(x, y_x) = 1$. 

\end{definition}

In this case, we denote $f: X \cdots \to Y$, as a fuzzy map and $y_x$ is called  fuzzy image of $x$.  We asserts that two fuzzy maps $f$ and $g$ on a set are equal, i.e. $f \equiv g$,  if and only if for each $x \in X$, both $f$ and $g$ have same fuzzy images. 
\begin{definition}Bijective fuzzy map: Let $f: X \cdots \to Y$ be a fuzzy map. Then $f$ is said to be one-one, if for any $x_1, x_2 \in X$ and $y \in Y$, whenever $f(x_1, y)=f(x_2, y) = 1$, then $x_1 = x_2$. Further, $f$ is   called onto if for each $y \in Y$, there exists $x \in X$ such that $f(x, y) = 1$. A fuzzy map is bijective if it is one-one as well as onto. 
\end{definition}
\begin{definition}
Composition of fuzzy maps: Let $f$ and $g$ be fuzzy subsets of the set $X\times Y$ and $Z \times X$ respectively where $X, Y$ and $Z$ are non empty sets. Then  $$f\circ g: Z\times Y\to[0, 1]\ \text{defined as}\ (f\circ g)(z, y) =\sup_{a \in X}\big\{ f(a, y)\ \big| \ g(z, a) = 1\big\}.$$
If there is no such $a \in X$ such that $g(z, a) = 1$ for a given $z \in Z$, then as in $[7, \ \text{Chapter}\ 1]$, we assume $(f\circ g)(z, y) =0$.
\end{definition}

\begin{remark} If  $g$ is a fuzzy map on  $Z \times X$, then $$(f\circ g)(z, y)= f(a, y)\ \ \text{where}\ \ g(z, a) = 1.$$\end{remark}

\begin{definition}Fuzzy subgroup: 
Let $G$ be a group and $\mu \in \mathcal{F}\mathcal{P}(G)$. Then $\mu$ is called a fuzzy subgroup of $G$ if for all $x, y \in G$, we have   $\mu(xy) \geq \mu(x)\wedge \mu(y$) and $\mu(x^{-1})\geq \mu(x)$.  Further, the set of all fuzzy subgroups of $G$ is denoted by  $\mathcal{F}(G)$.

\end{definition}
\begin{definition}Normal fuzzy  subgroup: Let $\mu \in  \mathcal{F}(G)$. Then $\mu$ is called a normal fuzzy  subgroup of $G$ if $\mu(xy) = \mu(yx)$ for all $x, y \in G$.

\end{definition}
\begin{definition}Fuzzy homomorphism: A fuzzy map $f: G \cdots \to G'$ is called a fuzzy homomorphism if for all $ x_1, x_2 \in G$ and $y \in G'$, we have $$f(x_1x_2, y) = \sup_{y_1, y_2 \in G'}\big\{f(x_1, y_1)\wedge f(x_2, y_2)\ \big| \ y = y_1y_2\big\}.$$

\end{definition}

\begin{definition}Kernel of a fuzzy homomorphism: Let $f: G\cdots \to G'$ be  a fuzzy homomorphism. Then we define $$K = \text{Ker}f = \big\{x \in G\ \big|\ f(x, e') = 1\big\}.$$\end{definition}
Next theorem is a basic result on fuzzy homormorphism. We refer to $[6]$ for the proofs.
\begin{theorem}
Let $f:G\cdots \to G'$ be a fuzzy homomorphism. Then \begin{enumerate}
\item $y_{x_1x_2} = y_{x_1}\cdot y_{x_2}$ where $y_{x_i}\in G'$ denote the unique element corresponding to $x_i \in G$ for $1 \leq i \leq 2$ such that $f(x_i, y_{x_i}) = 1$.  \item $f(e, e') = 1$ where $e$ and $e'$ denote the respective identities of $G$ and $G'$. 
\item $y_x^{-1} = y_{x^{-1}}$ for any $y_x \in G'$ and $x \in G$. 
\item $f(x, y) = 1 \implies f(x^{-1}, y^{-1}) = 1$ for any $x \in G$ and $y \in G'$. 

\end{enumerate}
\end{theorem}
Now, we prove the following fuzzy analogue of some results on kernel of a homomorphism.
\begin{theorem}
For a fuzzy homomorphism $f: G\cdots \to G'$,  kernel $K$ is a normal subgroup of $G$. Also $f$ is one-one if and only if $K = \{e\}$. 
\end{theorem}
\begin{proof}
By $(2)$ of Theorem $(2.1)$, we have $f(e, e') = 1$ which shows that $K$ is non empty. Let $x_1, x_2 \in K$. Then $$f(x_1x_2^{-1}, e') = \sup_{y_1, y_2 \in G'}\big\{f(x_1, y_1)\wedge f(x_2^{-1}, y_2)\ \big|\  e'=y_1y_2 \big\}$$ $$\geq f(x_1, e') \wedge f(x_2^{-1}, e') $$ $$ = f(x_1, e') \wedge f(x_2^{-1}, (e')^{-1}) = 1 \wedge 1 = 1 \ $$where above holds because of $(4)$ of Theorem $(2.1)$. Above shows that $x_1x_2^{-1}\in K$, i.e., $K$ is a subgroup of $G$. Now, let $g \in G$ and $x \in K$. Then $f(x, e') = 1$ and there exists a unique $y_g \in G'$ such that $f(g, y_g) = 1$. Now $$f(g^{-1}xg, e') = \sup_{y_1, y_2 \in G'}\big\{f(g^{-1}x, y_1)\wedge f(g, y_2)\ \big|\ e'=y_1y_2 \big\}\geq f(g^{-1}x, y_g^{-1}) \wedge f(g, y_g) $$ $$= f(g^{-1}x, y_g^{-1}) \wedge 1 =  f(g^{-1}x, y_g^{-1})  = \sup_{y_1, y_2 \in G'}\big\{f(g^{-1}, y_1)\wedge f(x, y_2)\ \big|\  y_g^{-1}=y_1y_2 \big\}$$ $$\geq f(g^{-1}, y_g^{-1})\wedge f(x, e') = 1 \wedge 1 = 1.$$ Above shows that $K$ is a normal subgroup of $G$. Remaining part of the theorem is lucid.
\end{proof}

\section{fuzzy automorphism and fuzzy inner automorphism of a group }

A fuzzy homomorphism $f: G\cdots \to G$ is said to be fuzzy automorphism provided it is one-one and onto. Now, in this section we discuss our main results on fuzzy automorphism and inner automorphism of a group $G$.

\begin{lemma}
Composition of fuzzy automorphisms of a group is again a fuzzy automorphism.

\end{lemma}

\begin{proof}
Let $f$ and $g$ be two fuzzy automorphisms of $G$. Then  for each $y \in G$ and $x_1, x_2 \in G$, $$f(x_1x_2, y) = \sup_{y_1, y_2 \in G}\big\{f(x_1, y_1)\wedge f(x_2, y_2)\ \big|\ y= y_1y_2\big\}$$ and $$g(x_1x_2, y) = \sup_{y_1, y_2  \in G}\big\{g(x_1, y_1)\wedge g(x_2, y_2)\ \big|\ y= y_1y_2\big\}.$$ First we show that $f\circ g$ is a homomorphism. Since $f$ and $g$ are fuzzy maps, there exist unique $y_i,\ z_i$, for $1 \leq i \leq 2$, such that $$f(x_i, y_i)  = 1= g(x_i, z_i) \ \text{for}\ 1\leq i \leq 2.$$Now we claim that $g(x_1x_2, z_1z_2) = 1$. Observe that $$g(x_1x_2, z_1z_2) = \sup_{a, b \in G}\big\{g(x_1, a)\wedge g(x_2, b)\ \big| \  z_1z_2 = ab\big\}$$ $$\geq g(x_1, z_1)\wedge g(x_2, z_2) = 1.$$ Thus, claim holds. Now because of Definition $2.4$, we have \begin{equation}
(f\circ g)(x_1x_2, y) = f(z_1z_2, y)= \sup_{y_1, y_2 \in G}\big\{ f(z_1, y_1)\wedge f(z_2, y_2) \ \big|\ y = y_1y_2\big\}. \end{equation}Now as $$(f\circ g)(x_i, y_i) = \sup_{a \in G}\big\{f(a, y_i) \ \big|\ g(x_i, a) = 1\big\}  = f(z_i, y_i),\ 1\leq i \leq 2,$$we have \begin{equation*}
\sup_{y_1, y_2 \in G}\big\{ (f\circ g)(x_1, y_1)\wedge (f\circ g)(x_2, y_2)\ \big| \ y = y_1y_2\big\}
\end{equation*} $$ = \sup_{y_1, y_2 \in G}\big\{ f(z_1,  y_1)\wedge f(z_2, y_2)\ \big| \ y = y_1y_2\big\}$$ $$ = (f\circ g)(x_1x_2, y)$$ where above holds by incorporating equation $(3.1)$. Thus, $f\circ g$ is a fuzzy homomorphism. 
Now, we show that $f\circ g$ is a bijection. For any $y \in G$,  take $x_1, x_2 \in G$ such that $$(f\circ g)(x_1, y) = (f\circ g)(x_2, y) = 1$$ By definition, above can be written as $$\sup_{a \in G}\big\{ f(a, y)\ \big|\ g(x_1, a) =1\big\}= \sup_{b \in G}\big\{ f(b, y)\ \big|\  g(x_2, b)=1 \big\}=1$$ $$\implies \ f(a, y) = f(b, y) = 1,$$ where $a$ and $b$ are the unique elements corresponding to $x_1$ and $x_2$ in $G$ such that $g(x_1, a) = g(x_2, b) = 1$. Above implies that $a = b$ as $f$ is one-one. This further implies that $g(x_1, a) = g(x_2, a)=1$ and as $g$ is one-one, this means $x_1 = x_2$ which completes the proof for $f\circ g$ to be one-one. Now we  show that $f\circ g$ is onto. Take any $y \in G$. We have to prove the existence of some $x \in G$ such that $(f\circ g)(x, y) = 1$. Since $f$ is onto, this means there exists some $a \in G$ such that $f(a, y) = 1$. Now as $g$ is onto, so there exist some $x_a \in G$ such that $g(x_a, a) = 1$. Then observe that $$(f\circ g)(x_a, y) = \sup_{b \in G}\big\{ f(b, y)\ \big|\  g(x_a, b)=1 \big\}.$$ As $g$ is a fuzzy map and $g(x_a, a) = 1 = g(x_a, b)$ implies that $a = b$. So above can be written as $$(f\circ g)(x_a, y) =f(a, y) = 1.$$
Hence, result.
\end{proof} 

\begin{lemma}
Associativity: Let $f, g$ and $h$ are fuzzy automorphisms of a group $G$. Then \begin{equation}
(f\circ g)\circ h\equiv f \circ (g\circ h).\end{equation}

\end{lemma}
\begin{proof}
Let $x, y \in G$. Then there exists $u \in G$ such that $h(x, u) = 1$. Again, $u \in G$ implies that there exists unique $z \in G$ such that $g(u, z) = 1$. Now, $$((f\circ g)\circ h)(x, y) = \sup_{a \in G}\big\{ (f\circ g)(a, y) \ \big|\ h(x, a) = 1\big\} $$ $$ = (f\circ g)(u, y) = f(z, y).$$  Now consider the right side of equation $(3.2)$. We have\begin{equation}
 (f \circ (g\circ h))(x, y) = \sup_{a \in G}\big\{ f(a, y) \ \big|\ (g\circ h)(x, a) = 1\big\}. 
\end{equation} Futher,  observe that $$(g\circ h)(x, z)=\sup_{b \in G}\big\{ g(b, z) \ \big|\  h(x, b) = 1\big\} $$ $$= g(u, z)=1. $$  So equation $(3.3)$ yields $$ (f \circ (g\circ h))(x, y) = f(z, y).$$ In particular, fuzzy images of $(f\circ g)\circ h$ and $ f \circ (g\circ h)$ are same. Hence $(f\circ g)\circ h\equiv f \circ (g\circ h)$.
\end{proof}

Now, let us consider a fuzzy homomorphism $I: G \cdots \to G$ such that $I(x, y) = 1$ if and only if $x = y$. It is easy to verify that $I$ is one-one and onto. Existence of such a  map is shown in the next section. In the next lemma we show that $I$ is infact an identity map.
\begin{lemma}
For any fuzzy automomorphism $f$, we have $$f\circ I \equiv f \equiv I\circ f.$$

\end{lemma}
\begin{proof}
Let $x \in G$ and $y$ be the fuzzy image of $x$ under $f$. Then $$(f\circ I)(x, y) = \sup_{a \in G}\big\{f(a, y)\ | \ I(x, a) = 1\big\} = f(x, y)=1.$$ Above implies that $y$ is fuzzy image of $x$ under $f\circ I$.
Also $$(I\circ f)(x, y) = \sup_{a \in G}\big\{I(a, y)\ | \ f(x, a) = 1\big\}.$$But as $y$ is the fuzzy image of $x$ under $f$, above can  be written as $$(I\circ f)(x, y) = I(y, y) = 1$$ which means $y$  is fuzzy image of $x$ under $I\circ f$. Therefore, result. 
\end{proof}

Now we define the notion of inverse fuzzy map. Let  $f: G\cdots \to G$ be a one-one and onto fuzzy homomorphism. Define $g : G\times  G \to [0, 1]$ by $g(y, x) = f(x, y)$. 
\begin{lemma}
$g$ is a well defined bijective fuzzy map and $$g\circ f \equiv I \equiv f\circ g.$$

\end{lemma}

\begin{proof}
Suppose there exist $x_1, x_2 \in G$ corresponding to some  $y \in G$ such that $g(y, x_1) = g(y, x_2) = 1$. By definition, this means $f(x_1, y) = 1 = f(x_2, y)$. But as $f$ is one-one, we must have $x_1 = x_2$. This shows that $g$ is a fuzzy map. It is easy to check that $g$ is a bijective fuzzy map. Now take any $x \in G$ and $y_1, y_2$ be its fuzzy images under $f\circ g$ and $I$ respectively. Then, $(f\circ g)(x, y_1)= 1 = I(x, y_2)$. We need to show $y_1 = y_2$. $I(x, y_2)  =1 $ implies that $x = y_2$. Now, by definition $$(f\circ g)(x, y_1) = \sup_{a \in G}\big\{f(a, y_1)\ | \ g(x, a) = 1\big\} = 1.$$As $g$ is a fuzzy map, there exist a unique  $a\in G$ such that $$(f\circ g)(x, y_1) = f(a, y_1) = 1 \ \text{and}\ g(x, a) = f(a, x) = 1.$$But as $f$ is one-one, above implies $x = y_1$ and hence $y_1 = y_2$. Therefore, $f\circ g\equiv I$. Similarly, we can show that $g\circ f \equiv I$.
\end{proof}
 
\begin{lemma} Map $g\circ f$ is a fuzzy homomorphism.

\end{lemma}
\begin{proof}
We need to show that $$(g\circ f)(x_1x_2, z) = \sup_{z_1, z_2 \in G}\big\{ (g\circ f)(x_1, z_1) \wedge (g\circ f)(x_2, z_2)\ \big|\  z = z_1z_2 \big\}$$ for $x_1, x_2$ and $z \in G$. So, let $x_1, x_2 \in G$. Then there exist unique $y_{x_1}, y_{x_2} \in G$ such that $$f(x_1, y_{x_1}) = 1 = f(x_2, y_{x_2}).$$Since $f$ is a fuzzy homomorphism, we have $f(x_1x_2, y_{x_1}y_{x_2}) = 1$. Now, let $z_1, z_2 \in G$ be such that $z = z_1z_2$. Then $$(g\circ f)(x_1x_2, z) = (g\circ f)(x_1x_2, z_1z_2) = g(y_{x_1}y_{x_2}, z_1z_2)$$ $$\hspace{20mm} = f(z_1z_2, y_{x_1}y_{x_2}) = \sup_{a, b \in G}\big\{ f( z_1, a) \wedge f( z_2, b)\ \big|\  y_{x_1}y_{x_2}=ab\big\}$$ $$ \hspace{15mm}\geq f( z_1, y_{x_1}) \wedge f( z_2, y_{x_2}) = g(  y_{x_1}, z_1) \wedge g(  y_{x_2}, z_2) $$ $$ = (g\circ f)(x_1, z_1) \wedge (g\circ f)(x_2, z_2).$$ Since $z_1, z_2$ are arbitrary, we have $$(g\circ f)(x_1x_2, z) \geq \sup_{z_1, z_2 \in G}\big\{ (g\circ f)(x_1, z_1) \wedge (g\circ f)(x_2, z_2)\ \big|\  z = z_1z_2 \big\}.$$For the other way around, since $f$ is a fuzzy homormorphism, by $(4)$ of Theorem $2.1$  $$f(x, y_x) = 1 \implies f(x^{-1}, y_{x^{-1}}) = 1.$$Consider $$\sup_{z_1, z_2 \in G}\big\{ (g\circ f)(x_1, z_1) \wedge (g\circ f)(x_2, z_2)\ \big|\  z = z_1z_2 \big\} \geq  (g\circ f)(x_1, x_1) \wedge (g\circ f)(x_2, x_1^{-1} z)$$ $$ = 1 \wedge (g\circ f)(x_2, x_1^{-1} z) = g(y_{x_2},x_1^{-1} z) = f(x_1^{-1} z, y_{x_2}) $$ $$ = \sup_{a, b \in G}\big\{ f( x_1^{-1}, a) \wedge f( z, b)\ \big|\ y_{x_2}=ab\big\}$$ $$\geq f( x_1^{-1}, y_{x_1}^{-1}) \wedge f( z, y_{x_1}y_{x_2})  =f( z, y_{x_1}y_{x_2}) = g(y_{x_1}y_{x_2}, z)$$ $$ = (g\circ f)(x_1x_2, z).$$Thus, result.
\end{proof}
\begin{lemma}
$g$ is a fuzzy automorphism.  
\end{lemma}
\begin{proof}
We just need to show that $g$ is a fuzzy homomorphism, i.e. $$g(y_1y_2, z) = \sup_{z_1, z_2 \in G}\big\{g(y_1, z_1)\wedge g(y_2, z_2)\ | \ z = z_1z_2\big\} $$ for any $y_1, y_2$ and  $z \in G$. Take $y_1, y_2, \in G$, this means there exist $x_1, x_2 \in G$ such that $f(x_1, y_1) = 1 = f(x_2, y_2)$ as $f$ is onto. Since $f$ is a fuzzy homomorphism, so $f( x_1x_2, y_1y_2) = 1$. For $z = z_1z_2$, we have $$(g\circ f)(x_1, z_1) = g(y_1, z_1)\ \text{and} \ (g\circ f)(x_2, z_2) = g(y_2, z_2).$$Now, consider 
$$\sup_{z_1, z_2 \in G}\big\{g(y_1, z_1)\wedge g(y_2, z_2)\ | \ z = z_1z_2\big\} =  \sup_{z_1, z_2 \in G}\big\{(g\circ f)(x_1, z_1)\wedge (g\circ f)(x_2, z_2)\ | \ z = z_1z_2\big\}$$ $$ = (g\circ f)(x_1x_2, z) = g(y_1y_2, z) $$ where above holds because of Lemma $3.5$. Above means that $g$ is a fuzzy homomrophism. Lemmas $3.4$ and $3.6$ yields that $g = f^{-1}$, i.e. $g$ is inverse  of fuzzy automorphism $f$.

\end{proof}
 
Let $S$ be a non-empty set. By $A_{F}(S)$ we denote the set of all fuzzy bijective maps on the set $S$. From the above discussion one can easily deduce that $A_{F}(S)$ is a group w.r.t. composition of fuzzy mappings. This group is termed as fuzzy permutation group on the set $S$. For a group $G$, Let $A_{F}(G)$ be the fuzzy permutation group on the group $G$  and by $Aut_{F}(G)$ we denote the set of all fuzzy automorphisms on a group $G$. Then $A_{F}(G)$ and $Aut_{F}(G)$ form groups under composition of fuzzy mappings.

\section{fuzzy permutation induced by a fuzzy subgroup and Cayley Theorem }
\begin{lemma}
	Let $\mu \in F(G)$ such that $\mu(x)=1$ if and only if $x=e$. Then for $a\in G$, the map $f_{a}^{\mu}:G \times G \rightarrow [0,1]$ given by $ f_{a}^{\mu}(x,y)=\mu(x^{-1}a^{-1}y)$ is a fuzzy permutation on $G$. 
\end{lemma}
\begin{proof} Let $x_1, x_2 \in G$ such that for some $y \in G$, we have $f_{a}^{ \mu}(x_1, y) = f_{a}^{ \mu}(x_2, y) = 1$. This means $\mu(x_{1}^{-1}a^{-1}y) = \mu(x_{2}^{-1}a^{-1}y)=1$. But this holds only if $x_{1}^{-1}a^{-1}y = x_{2}^{-1}a^{-1}y =e$  and hence $x_1 = x_2$. Thus, $f_{a}^{ \mu}$ is one-one. For onto, let $y \in G$. Then $a^{-1}y$ is also in  $G$. Observe that $f_{a}^{ \mu}(a^{-1}y, y) = \mu(e) = 1$. Thus, $f_{a}^{ \mu}$ is onto. 
\end{proof}

{\begin{theorem} Let $\mu \in F(G)$ such that $\mu(x)=1$ if and only if $x=e$ and   $A_{F}^{\mu}(G)=\{f_{a}^{\mu} :a \in G\}$. Then $A_{F}^{\mu}(G)$ is a group under composition of fuzzy mappings.
\end{theorem}
\begin{proof} Let $A=A_{F}^{\mu}(G)$. Then for $f_{a}^{ \mu}, f_{b}^{ \mu}\in A $, we have
	\begin{eqnarray*}
		f_{a}^{ \mu}\circ f_{b}^{ \mu}(x,y)= f_{a}^{ \mu}(bx,y)= \mu(x^{-1}b^{-1} a^{-1}y)=f_{ab}^{\mu}(x,y)
	\end{eqnarray*}
As for any 	$f_{a}^{ \mu}$, 	$f_{a}^{ \mu}\circ I_{e}^{ \mu}= I_{e}^{ \mu}\circ f_{a}^{ \mu}= f_{a}^{ \mu}$, the fuzzy identity map $I=	I_{e}^{ \mu}$ is the identity element of  $A_{F}^{\mu}(G)$. It is easy to see that $f_{a}^{ \mu}\circ f_{a^{-1}}^{ \mu}= I_{e}^{ \mu}=f_{a^{-1}}^{ \mu}\circ f_{a}^{ \mu}$. 
\end{proof} 

\begin{theorem}  (Cayley Theorem) $G\cong A_{F}^{\mu}(G)$.
	
\end{theorem}
\begin{proof} Define a map $\sigma :G\longrightarrow A_{F}^{\mu}(G)\subseteq A_{F}(G)$ by 
	$\sigma (a)= f_{a}^{ \mu}$. Since $f_{a}^{ \mu}\circ f_{b}^{ \mu}=f_{ab}^{\mu}$, $\sigma$ is a homomorphism of groups. Let $a\in Ker\; \sigma$.Then $f_{a}^{ \mu}=	I_{e}^{ \mu}$ i.e., 
		$\mu(x^{-1}a^{-1}y)=\mu(x^{-1}y)$ $\forall (x,y)\in G\times G$. In particular, for $y=x$, $\mu(x^{-1}a^{-1}x)=\mu(e)=1$. This holds only when $x^{-1}a^{-1}x=e$. Eventually, $a=e$ and $\sigma$ is a monomorphism of groups. For any $f_{a}^{ \mu} \in A_{F}^{\mu}(G)$, $\sigma(a)=f_{a}^{ \mu}$. Thus $\sigma$ is an isomorphism and hence $G\cong A_{F}^{\mu}(G)$.  
 
\end{proof}
 \begin{definition}
 Fuzzy inner automorphism: A fuzzy map $f:G \cdots \to G$ is said to be a class preserving map if $f(x,y)=1$ if and only if $y=a^{-1}xa$ for some $a\in G$. Let Aut$_F(G)$ denotes the group of all fuzzy automorphisms of $G$, then we say that $f \in$ Aut$_F(G)$ is an inner automorphism if and only if there exists $g \in G$ (fixed) such that $f(x, y) = 1$ if and only if $y = g^{-1}xg$. In this case, we denote $f$ by $f_g$. Further, let Inn$_F(G)$ denotes the set of all fuzzy inner automorphisms of $G$. 
 \end{definition}
 \begin{lemma}
Composition of fuzzy inner automorphisms is again a fuzzy inner automorphism.
\end{lemma}
\begin{proof}
Let $f_{g_1}$ and $f_{g_2}$ be two fuzzy inner automorphisms of $G$. Then $$(f_{g_1}\circ f_{g_2})(x, y)= \sup_{a \in G}\big\{f_{g_1}(a, y)\ \big| \ f_{g_2}(x, a)=1\big\}= f_{g_1}(g_2^{-1}xg_2, y) .$$Now  from above, $$(f_{g_1}\circ f_{g_2})(x, y)=1 \Longleftrightarrow y = g_1^{-1}g_2^{-1}xg_2g_1.$$Similarly, by definition $$f_{g_2g_1}(x, y) = 1 \Longleftrightarrow y = g_1^{-1}g_2^{-1}xg_2g_1.$$Thus fuzzy images of $f_{g_1}\circ f_{g_2}$ and $f_{g_2 g_1}$ are same and hence it follows that $f_{g_1}\circ f_{g_2}\equiv f_{g_2g_1}$.
\end{proof}

\begin{lemma}
If $f_g \in $Inn$_F(G)$, then $f_g^{-1} $ is also in  Inn$_F(G)$ for any $g \in G$ and $f_g^{-1} \equiv f_{g^{-1}}$.

\end{lemma}
\begin{proof}
Take $x$ and $y$ in $G$, then by definition of  inverse fuzzy map, we have $$f_g^{-1}(x, y) = 1 \Longleftrightarrow f_g(y, x) = 1   \Longleftrightarrow y = (g^{-1})^{-1}xg^{-1}.   $$Thus, $f_g^{-1} $ is a fuzzy inner automorphism. Also note that $$f_{g^{-1}}(x, y) = 1 \Longleftrightarrow y = (g^{-1})^{-1}xg^{-1}.   $$This shows that $f_g^{-1} \equiv f_{g^{-1}}$.
\end{proof}
\begin{lemma}
Inn$_F(G)$ is a normal subgroup of  Aut$_F(G)$.

\end{lemma}
\begin{proof}
Lemmas $3.7$ and $3.8$ together show that Inn$_F(G)$ is a  subgroup of  Aut$_F(G)$. Now, let $f$ be any element in Aut$_F(G)$ and $f_g$ be in Inn$_F(G)$. Then we have $$(f^{-1}\circ f_g\circ f)(x, y) = \sup_{a \in G}\big\{(f^{-1}\circ f_g)(a, y)\ \big| \ f(x, a)=1\big\}.$$Since $f \in$  Aut$_F(G)$, so for any $x \in G$, there exists a unique $z \in G$ such that $f(x, z) = 1$. Then $$(f^{-1}\circ f_g\circ f)(x, y) =(f^{-1}\circ f_g)(z, y) = f^{-1}(g^{-1}zg, y).$$As $g \in G$, there exists a unique $a \in G$ such that $f^{-1}(g, a) = 1$ which further implies that $f^{-1}(g^{-1}, a^{-1}) = 1$. Now $$f^{-1}(g^{-1}zg, a^{-1}xa) = \sup_{y_1, y_2\in G}\big\{f^{-1}(g^{-1}z, y_1)\wedge f^{-1}(g, y_2)\ \big| \  a^{-1}xa=y_1y_2\big\}$$ $$\geq f^{-1}(g^{-1}z, a^{-1}x)\wedge f^{-1}(g, a) = f^{-1}(g^{-1}z, a^{-1}x)$$ $$ = \sup_{y_1, y_2\in G}\big\{f^{-1}(g^{-1}, y_1)\wedge f^{-1}(z, y_2)\ \big| \  a^{-1}x=y_1y_2 \big\}$$ $$\geq f^{-1}(g^{-1}, a^{-1})\wedge f^{-1}(z, x) = f^{-1}(g^{-1}, a^{-1}) = 1.$$This implies that $(f^{-1}\circ f_g\circ f)(x, a^{-1}xa) = 1$ and hence $f^{-1}\circ f_g\circ f \in$  Inn$_F(G)$. Thus, Inn$_F(G)$ is a normal subgroup of Aut$_F(G)$.
\end{proof}

\section{fuzzy   inner automorphism of a group G induced by a fuzzy normal subgroup}
In this section, we turn our attention to define fuzzy inner automorphism induced by a fuzzy normal subgroup. These automorphisms  can  also be seen as  typical examples of the inner automorphisms discussed in Section $3$. 
Let $G$ be a group and  $\mu$ is a normal fuzzy subgroup of $G$ and $\mu(x) = 1$ if and only if $x = e$. Let $g \in G$ and define $f_{g}^{ \mu}:G\cdots \to G$ by $f_{g}^{ \mu}(x, y) = \mu(x^{-1}gyg^{-1})$. We claim that $f_{g}^{ \mu}$ is a fuzzy map. Observe that $$f_{g}^{\mu}(x, g^{-1}xg) = \mu(x^{-1}gg^{-1}xgg^{-1}) = \mu(e) = 1.$$ Let if possible $f_{g}^{ \mu}(x, z) = 1$ for some $z \in G$. Then this means $$\mu(x^{-1}gzg^{-1}) = 1 \implies x^{-1}gzg^{-1}= e \implies z = g^{-1}xg.$$Thus $f_{g}^{ \mu}$ associates a unique $g^{-1}xg$ for each $x \in G$ such that $f_{g}^{ \mu}(x, g^{-1}xg) = 1$. Thus, claim holds. 
\begin{lemma}
For any $g \in G$, $f_{g}^{ \mu}$ is a fuzzy homomorphism.

\end{lemma}
\begin{proof}
Let $x_1, x_2$ and $y \in G$ such that $y = y_1y_2$ for $y_1, y_2 \in G$.  Then $$f_{g}^{ \mu}(x_1x_2, y)=f_{g}^{ \mu}(x_1x_2, y_1y_2) = \mu(x_2^{-1}x_1^{-1}gy_1y_2g^{-1}) = \mu(x_2^{-1}x_1^{-1}gy_1(g^{-1}g)y_2g^{-1})$$ $$\hspace{12mm} = \mu(x_2^{-1}x_1^{-1}gy_1g^{-1}z)\quad (\text{where}\ z^{-1}gy_2g^{-1} = e)$$  $$\hspace{27mm} =\mu(x_2^{-1}zz^{-1}x_1^{-1}gy_1g^{-1}z) \geq \mu(x_2^{-1}z)\wedge \mu(z^{-1}x_1^{-1}gy_1g^{-1}z).$$Now as $\mu$ is normal, so above implies $$f_{g}^{ \mu}(x_1x_2, y) \geq \mu(x_2^{-1}z)\wedge \mu(x_1^{-1}gy_1g^{-1})\hspace{20mm}$$ $$ \hspace{40mm}= \mu(x_2^{-1}gy_2g^{-1})\wedge \mu(x_1^{-1}gy_1g^{-1})= f_{g}^{ \mu}(x_1, y_1)\wedge f_{g}^{ \mu}(x_2, y_2).$$ This means that $f_{g}^{ \mu}(x_1x_2, y) \geq \ f_{g}^{ \mu}(x_1, y_1)\wedge f_{g}^{ \mu}(x_2, y_2)$ for all $y_1, y_2 \in G$ such that $y = y_1y_2$ which further implies that $$f_{g}^{ \mu}(x_1x_2, y) \geq\sup_{y_1, y_2 \in G}\big\{ f_{g}^{ \mu}(x_1, y_1)\wedge f_{g}^{ \mu}(x_2, y_2)\ \big|\ y = y_1y_2\big\}.$$Now  consider the other way around. Observe that $$\sup_{y_1, y_2 \in G}\big\{ f_{g}^{ \mu}(x_1, y_1)\wedge f_{g}^{ \mu}(x_2, y_2)\ \big|\ y = y_1y_2\big\}\geq  f_{g}^{ \mu}(x_1, g^{-1}x_1g)\wedge f_{g}^{ \mu}(x_2, g^{-1}x_1^{-1}gy)$$ $$ = \mu(x_1^{-1}g(g^{-1}x_1g)g^{-1}) \wedge \mu(x_2^{-1}g(g^{-1}x_1^{-1}gy)g^{-1})$$ $$ = \mu(e) \wedge \mu(x_2^{-1}x_1^{-1}gyg^{-1}) = \mu(x_2^{-1}x_1^{-1}gyg^{-1})$$ $$ = f_{g}^{ \mu}(x_1x_2, y).$$Thus, combining above results yields $$f_{g}^{ \mu}(x_1x_2, y) = \sup_{y_1, y_2 \in G}\big\{ f_{g}^{ \mu}(x_1, y_1)\wedge f_{g}^{ \mu}(x_2, y_2)\ \big|\ y = y_1y_2\big\}.$$Thus, $f_{g}^{ \mu}$ is a fuzzy homomorphism. 

\end{proof}
\begin{lemma}
For any $g \in G$, $f_{g}^{ \mu}$ is a one-one onto class preserving fuzzy homomorphism.
\end{lemma}
\begin{proof}
First, let $f_{g}^{ \mu}(x, y) = 1$. This means that $
\mu(x^{-1}gyg^{-1}) = 1 \implies y = g^{-1}xg$. In the reverse direction, we have $f_{g}^{ \mu}(x, g^{-1}xg)= \mu(x^{-1}g(g^{-1}xg)g^{-1}) = \mu(e) = 1$. This shows that $f_{g}^{ \mu}$ is a  class preserving map. Now, let $x_1, x_2 \in G$ such that for some $y \in G$, we have $f_{g}^{ \mu}(x_1, y) = f_{g}^{ \mu}(x_2, y) = 1$. This means $y = g^{-1}x_1g = g^{-1}x_2g$ and hence $x_1 = x_2$. Thus, $f_{g}^{ \mu}$ is one-one. For onto, let $y \in G$. Then $gyg^{-1}$ is also in  $G$. Observe that $f_{g}^{ \mu}(gyg^{-1}, y) = \mu(e) = 1$. Thus, $f_{g}^{ \mu}$ is onto. 

\end{proof}
\begin{lemma}
Let $f_{g_1}^{ \mu}$ and $f_{g_2}^{ \mu}$ be two fuzzy inner automorphisms of $G$ induced by $\mu$. Then $f_{g_1}^{ \mu}\circ f_{g_2}^{ \mu}$ is again a fuzzy inner automorphism of $G$ induced by $\mu$ and $f_{g_1}^{ \mu}\circ f_{g_2}^{ \mu} = f_{g_2g_1}^{ \mu}$ and in particular $f_{g_1}^{ \mu}\circ f_{g_2}^{ \mu} \equiv  f_{g_2g_1}^{ \mu}$ . 
\end{lemma}
\begin{proof}
We have $$(f_{g_1}^{ \mu}\circ f_{g_2}^{ \mu})(x, y) = \sup_{a \in G}\big\{f_{g_1}^{ \mu}(a, y)\ \big|\ f_{g_2}^{ \mu}(x, a) = 1\big\} = f_{g_1}^{ \mu}(g_2^{-1}xg_2, y)$$ $$ = \mu(g_2^{-1}x^{-1}g_2g_1yg_1^{-1}) = \mu(x^{-1}(g_2g_1)y(g_2g_1)^{-1}) \quad (\mu\ \text{is normal})$$ $$ = f_{g_2g_1}^{ \mu}(x, y).$$ 
since $ f_{g_2g_1}^{ \mu}(x, y)$ is a fuzzy inner automorphism induced by $\mu$ $$\implies f_{g_1}^{ \mu}\circ f_{g_2}^{ \mu}$$  is a fuzzy inner  automorphism induced by $\mu$.
Hence result.
\end{proof}
\begin{lemma}
Let $f_{g_1}^{ \mu}$, $f_{g_2}^{ \mu}$ and $f_{g_3}^{ \mu}$ be fuzzy inner automorphisms of $G$ induced by $\mu$. Then $$(f_{g_1}^{ \mu}\circ f_{g_2}^{ \mu})\circ f_{g_3}^{ \mu}\equiv f_{g_1}^{ \mu}\circ (f_{g_2}^{ \mu}\circ f_{g_3}^{ \mu}) \equiv f_{g_3g_2g_1}^{ \mu}. $$ 
\end{lemma}
\begin{proof}
Proof follows directly from Lemma $4.3$. 
\end{proof}
\begin{lemma} $I_{e}^{ \mu}(x, y) = \mu(x^{-1}y)$ is a fuzzy identity inner automorphism induced by  $\mu$. Here $e$ denotes the identity of group $G$.
\end{lemma}
\begin{proof}
Let $f_{g}^{ \mu}$ be any fuzzy inner automomorphism. Then note that $(f_{g}^{ \mu}\circ  I_{e}^{ \mu})(x, y) = f_{g}^{\mu}(x, y)$ which means $(f_{g}^{ \mu}\circ  I_{e}^{ \mu})\equiv f_{g}^{\mu}$. Also $$(I_{e}^{ \mu}\circ f_{g}^{ \mu})(x, y) = I_{e}^{ \mu}(g^{-1}xg, y) = \mu(g^{-1}x^{-1}geye^{-1}) = \mu(g^{-1}x^{-1}gy) = \mu(x^{-1}gyg^{-1}) = f_{g}^{ \mu}(x, y)$$ $$ \implies I_{e}^{ \mu}\circ f_{g}^{ \mu} \equiv I_{e}^{ \mu}.$$ Thus, $I_{e}^{ \mu}$ is an identity map. Now, we show that $I_{e}^{ \mu}$ is a fuzzy homomorphism. Let $x_1, x_2 \in G$ and $y \in G$ such that $y = y_1y_2$ for $y_1, y_2 \in G$. Then  we have $$I_{e}^{ \mu}(x_1x_2, y) = I_{e}^{ \mu}(x_1x_2, y_1y_2) = \mu(x_2^{-1}x_1^{-1}y_1y_2)= \mu(x_1^{-1}y_1y_2x_2^{-1})$$ $$\geq \mu(x_1^{-1}y_1)\wedge \mu(y_2x_2^{-1}) = I_{e}^{ \mu}(x_1, y_1)\wedge I_{e}^{ \mu}(x_2, y_2).$$Since $y_1$ and $y_2$ are arbitrary, we have $$I_{e}^{ \mu}(x_1x_2, y) \geq \sup_{y_1, y_2 \in G}\big\{I_{e}^{ \mu}(x_1, y_1)\wedge I_{e}^{ \mu}(x_2, y_2)\ \big|\ y = y_1y_2\big\}.$$Further, $$\sup_{y_1, y_2 \in G}\big\{I_{e}^{ \mu}(x_1, y_1)\wedge I_{e}^{ \mu}(x_2, y_2)\ \big|\ y = y_1y_2\big\} \geq I_{e}^{ \mu}(x_1, x_1)\wedge I_{e}^{ \mu}(x_2, x_1^{-1}y)$$ $$ = I_{e}^{ \mu}(x_2, x_1^{-1}y) = \mu(x_2^{-1}x_1^{-1}y) = I_{e}^{ \mu}(x_1x_2, y).$$ Therefore $I_{e}^{ \mu}$ is a fuzzy homomorphism. It is easy to see that $I_{e}^{ \mu}$ is one-one and onto. Hence result.
 
\end{proof}
\begin{remark}
Fuzzy map $I_{e}^{ \mu}$ is a typical example of identity map discussed in Lemma $3.3$. 
\end{remark}
\begin{lemma}
For any $g \in G$, $(f_{g}^{ \mu})^{-1} \equiv f_{g^{-1}}^{ \mu}$. 
\end{lemma}
\begin{proof}
Observe that $$(f_g^{ \mu}\circ f_{g^{-1}}^{ \mu})(x, y) =  \sup_{a \in G}\big\{f_{g}^{ \mu}(a, y)\ \big|\ f_{g^{-1}}^{ \mu}(x, a) = 1\big\} = f_{g}^{ \mu}(gxg^{-1}, y)$$  $$ = \mu(gx^{-1}g^{-1}gyg^{-1}) = \mu(gx^{-1}yg^{-1})= \mu(yg^{-1}gx^{-1})$$ $$ =\mu(yx^{-1}) = I_e^{ \mu}(x, y).$$ Similarly, we can show that $(f_{g^{-1}}^{ \mu}\circ f_{g}^{ \mu})(x, y) = I_e^{\mu}(x, y)$. Thus, $(f_g^{ \mu})^{-1} = f_{g^{-1}}^{ \mu}$ and hence $(f_{g}^{ \mu})^{-1} \equiv f_{g^{-1}}^{ \mu}$. 
\end{proof}
Next result is now straight forward.
\begin{theorem}
The set of all fuzzy inner automorphisms of $G$ induced by a normal subgroup $\mu$ forms  a group under composition. 
\end{theorem}

In group theory, Cayley theorem states that a group $G$ is isomorphic to a subgroup some permutation group. We prove the following fuzzy analogue of this result
\begin{theorem}
	For a group $G$, $G/N_{G}(\mu)$ is isomorphic to a subgroup of $A_{F}(G,\mu)$
\end{theorem}

 $G/Z(G)\cong$ Inn$(G)$ which is a well known result in group theory.  Let Inn$_{F}(G, \mu)$ denotes the group of all fuzzy inner automorphism of $G$ induced by a fuzzy normal subgroup $\mu$.
\begin{theorem}
For a group $G$, $G/Z(G)\cong$ Inn$_{F}(G, \mu)$ where $Z(G)$ denotes the center of $G$. 
\end{theorem}
\begin{proof}
Define  a map $$\zeta:G \to \ \text{Inn}_{F}(G, \mu): g \mapsto f_{g^{-1}}^{  \mu}.$$ First of all we show that $\zeta$ is a group homomorphism. Note that $$\zeta(g_1g_2) = f_{g_2^{-1}g_1^{-1}}^{ \mu} = f_{g_1^{-1}}^{  \mu}\circ f_{g_2^{-1}}^{  \mu} = \zeta(g_1)\circ \zeta(g_2)$$ which shows that  $\zeta$ is a group homomorphism. Clearly, $\zeta$ is onto. Now we show that Ker$(\zeta) = Z(G)$ where Ker$(\zeta)$ denotes the kernel of $\zeta$. Let $g \in$ Ker$(\zeta)$. This means $\zeta(g) = f_{g^{-1}}^{  \mu} = I_{e}^{ \mu}$. This shows that $ f_{g^{-1}}^{  \mu} \equiv I_{e}^{ \mu}$ and hence, fuzzy images of  $f_{g^{-1}}^{  \mu}$ and $I_{e}^{ \mu}$ are same i.e., $gxg^{-1} = x$ for all $x \in G$. Therefore we deduce that Ker$(\zeta) \subseteq Z(G)$. Now to prove the reverse inclusion, take any $z \in Z(G)$. Then $\zeta(z) = f_{z^{-1}}^{ \mu}$ but $$f_{z^{-1}}^{ \mu}(x, y) = \mu(x^{-1}z^{-1}yz) = \mu(x^{-1}y)= I_{e}^{ \mu}(x, y).$$
In particluar $f_{z^{-1}}^{  \mu}  \equiv I_{e}^{ \mu}$.  This shows that $Z(G)\subseteq$ Ker$(\zeta)$.  Thus,  Ker$(\zeta)= Z(G)$. So, by fundamental theorem of homomorphism, we have  $G/Z(G)\cong$ Inn$_{F}(G,  \mu)$.

\end{proof}
The above result is  analogue to its counterpart from the classical group theory as the isomorphism involved in it is usual group isomorphism. On the other hand, if we go for the fuzzy isomorphism $\underset{F}{\cong}$ between $G$ and Inn$_{F}(G,  \mu)$, then we have the following interesting result.
\begin{theorem}
For  a group $G$, we have $G \underset{F}{\cong}$ Inn$_{F}(G,  \mu)$. 
\end{theorem}
\begin{proof}
Define a map $$\theta : G \cdots  \to \text{Inn}_{F}(G,  \mu)\ \text{by}\ (a, f_{b}^{ \mu}) \mapsto \mu(a^{-1}b^{-1}).$$Observe that $\theta(a, f_{a^{-1}}^{ \mu}) = \mu(a^{-1}a) = 1.$ This means $\theta$ is a fuzzy map which maps $a$ to $f_{a^{-1}}^{ \mu}$ uniquely. Now we show that $\theta$ is a fuzzy homomorphism. Let $a_1, a_2 \in G$ and $f_{c}^{ \mu} \in \text{Inn}_{F}(G,  \mu)$. So, we need to show that $$\theta(a_1a_2, f_{c}^{ \mu}) = \sup_{f_{x}^{ \mu}, f_{y}^{ \mu} \in \text{Inn}_{F}(G,  \mu)}\big\{\theta(a_1, f_{x}^{ \mu})\wedge \theta(a_2, f_{y}^{ \mu})\ \big|\ f_{c}^{ \mu}=f_{x}^{ \mu}\circ f_{y}^{ \mu}  \big\}.$$ Now, for $f_{x}^{ \mu}, f_{y}^{  \mu} \in \text{Inn}_{F}(G,  \mu)$ such that $\ f_{x}^{  \mu}\circ f_{y}^{  \mu} = f_{c}^{ \mu}$ we have $$\theta(a_1a_2, f_{c}^{  \mu})=\theta(a_1a_2, f_{x}^{  \mu}\circ f_{y}^{  \mu}) =
\theta(a_1a_2, f_{yx}^{ \mu}) = \mu(a_2^{-1}a_1^{-1}x^{-1}y^{-1}) = \mu(y^{-1}a_2^{-1}a_1^{-1}x^{-1})\hspace{20mm}$$ $$\geq \mu(y^{-1}a_2^{-1})\wedge \mu(a_1^{-1}x^{-1})= \mu(a_2^{-1}y^{-1})\wedge \mu(a_1^{-1}x^{-1})= \theta(a_1, f_{x}^{ \mu})\wedge \theta(a_2, f_{y}^{ \mu}).$$Above implies that  $$\theta(a_1a_2, f_{c}^{ \mu}) \geq \sup_{f_{x}^{ \mu}, f_{y}^{ \mu} \in \text{Inn}_{F}(G,  \mu)}\big\{\theta(a_1, f_{x}^{ \mu})\wedge \theta(a_2, f_{y}^{ \mu})\ \big|\ f_{c}^{ \mu}= f_{x}^{\mu}\circ f_{y}^{ \mu}  \big\}.$$Now, for the other way around, observe that $$\sup_{f_{x}^{ \mu}, f_{y}^{ \mu} \in \text{Inn}_{F}(G,  \mu)}\big\{\theta(a_1, f_{x}^{ \mu})\wedge \theta(a_2, f_{y}^{ \mu})\ \big|\ f_{c}^{ \mu}= f_{x}^{\mu}\circ f_{y}^{ \mu}  \big\}  \hspace{70mm}$$ $$=\sup_{f_{x}^{ \mu}, f_{y}^{ \mu} \in \text{Inn}_{F}(G,  \mu)}\big\{\theta(a_1, f_{x}^{ \mu})\wedge \theta(a_2, f_{y}^{ \mu})\ \big|\ f_{c}^{ \mu}= f_{yx}^{\mu}  \big\} \geq \theta(a_1, f_{a_1^{-1}}^{ \mu})\wedge \theta(a_2, f_{ca_1}^{ \mu})$$ $$ = \mu(a_1^{-1}a_1)\wedge \mu(a_2^{-1}a_1^{-1}c^{-1}) = \mu((a_1a_2)^{-1}c^{-1}) =\theta(a_1a_2, f_{c}^{ \mu}).  $$ Thus, $\theta$ is a fuzzy homomorphism.
Further, $$\text{Ker}(\theta) = \{x \in G\ : \ \theta(x, I_{e}^{ \mu}) = 1\} = \{x \in G\ : \ \mu(x^{-1}e^{-1}) = 1\} = \{e\}.$$ This implies that $\theta$ is one-one. Now take any $f_{g}^{ \mu} \in $Inn$_{F}(G,  \mu)$. Then as $g \in G$ implies that $\theta(g^{-1},  f_{g}^{ \mu}) = \mu(gg^{-1}) = 1$ which means that $\theta$ is onto and hence $\theta$ is a fuzzy isomorphism. 
\end{proof}

\bibliographystyle{plain}

\end{document}